\documentclass[reqno]{amsart}

\usepackage[colorlinks=true, pdfstartview=FitV, linkcolor=blue,citecolor=red, urlcolor=blue]{hyperref}

\usepackage{amsmath,amsthm,amssymb,mathrsfs}
\usepackage{color}
\usepackage{hyperref}
\usepackage{url}
\usepackage{verbatim}

\newcommand{\bburl}[1]{\textcolor{blue}{\url{#1}}}

\newtheorem{thm}{Theorem}[section]
\newtheorem{cor}[thm]{Corollary}
\newtheorem{lem}[thm]{Lemma}
\newtheorem{prop}[thm]{Proposition}
\theoremstyle{definition}

\theoremstyle{definition}

\theoremstyle{remark}

\newcommand\be{\begin{equation}}
\newcommand\ee{\end{equation}}
\newcommand\ben{\begin{enumerate}}
\newcommand\een{\end{enumerate}}

\newcommand{\R}{\ensuremath{{\mathbb R}}}
\newcommand{\C}{\ensuremath{{\mathbb C}}}

\newcommand{\HH}{{\mathbb H}}
\newcommand{\OO}{{\mathcal O}}

\newcommand{\Mod}[1]{\ (\mathrm{mod}\ #1)}

\newcommand\widebar[1]{\mathop{\overline{#1}}}

\newenvironment{psmallmatrix}{\left(\begin{smallmatrix}}{\end{smallmatrix}\right)}

\numberwithin{equation}{section}

\title{A first Kronecker limit formula for Kleinian groups}

\author{Zihan Miao}
\email{zmiao@smith.edu}
\address{Smith College, Burton Hall 115, Northampton, MA 01063, USA}

\author{Anh Nguyen}
\email{anguyen29@smith.edu}
\address{Smith College, Burton Hall 115, Northampton, MA 01063, USA}

\author{Tian An Wong}
\email{tiananw@umich.edu}
\address{University of Michigan-Dearborn, 4901 Evergreen Rd, 2002 CASL Building, Dearborn, MI 48128, USA}



\begin{document}

\begin{abstract}We prove a first Kronecker limit formula for cofinite discrete subgroups of SL$(2,\mathbb{C})$, also called Kleinian groups, generalizing a method of Goldstein over SL$(2,\mathbb R)$. The proof uses the Fourier expansion of Eisenstein series, leading to an analogue of the logarithm of the Dedekind eta function, and whose transformation law produces a function analogous to one related to Dedekind sums. We derive several basic properties of this new function.\end{abstract}

\maketitle


\section{Introduction}
\subsection{The first Kronecker limit formula and Dedekind sums}

The first Kronecker limit formula states that for any point $z$ in the upper-half plane, 
\[
\lim_{s\to1}\Big(\sum_{\substack{c,d=-\infty\\ (c,d)\neq (0,0)}}^\infty y^s|cz + d|^{-2s} - \frac{\pi}{s-1}\Big) = 2\pi(\gamma - \log2 - \log(\sqrt{y}|\eta(z)|^2)),
\]
where $\gamma$ is  the Euler-Mascheroni constant, and $\eta(z)$ the Dedekind eta function. The eta function satisfies the transformation law
\be
\label{trans}
\log \eta\left(\frac{az+b}{cz+d}\right) = \log\eta(z)+\frac12\log(cz + d) + \pi \sqrt{-1} S(a,b,c,d), \quad \begin{pmatrix}
 a& b \\ c & d 
\end{pmatrix}\in \mathrm{SL}(2,\R),
\ee
where 
\[
S(a,b,c,d) = \frac{a+d}{12c}-\frac14\frac{c}{|c|}-s(c,d)
\]
and $s(c,d)$ is the classical Dedekind sum, which can be defined as
\[
s(c,d) = \sum_{k=1}^{|c|-1}B_1\left(\frac{\pi k}{d}\right)B_1\left(\frac{\pi kc}{d}\right),
\]
where $B_1$ is the first periodic Bernoulli function, and all branches of the logarithm are the principal branch. 

Using an analogue of the first Kronecker limit formula for Eisenstein series associated to cofinite (but not cocompact) discrete subgroups $\Gamma$ of PSL$(2,\R)$, also called Fuchsian groups of the first kind, Goldstein defined a family of half-integer weight automorphic forms $\eta_{\Gamma,i}(z)$ analogous to the Dedekind eta function for each cusp $i$ of $\Gamma$ \cite{GLD}. Then using the transformation property of $\eta_{\Gamma,i}$, Goldstein constructed a family functions $S_{\Gamma,i}$ analogous to the functions $S(a,b,c,d)$ attached to $\Gamma$ and proved that it satisfies properties similar to the latter. On the other hand, the convergence of the limit formula was not proved, so it was referred to as a formal limit formula.

\subsection{Extension to imaginary quadratic fields}

In this paper, we adapt Goldstein's method to derive a first Kronecker limit formula for Eisenstein series attached to cofinite discrete subgroups $\Gamma$ of PSL$(2,\C)$, also known as Kleinian groups, that are not cocompact. Let 
\[
u = z + rj, \qquad z\in\C,\ r>0\in\R,\ j^2=-1
\]
be a point in the upper-half space $\HH$, and let $s\in\C$. Also let $\zeta_1,\dots,\zeta_h$ be the set of inequivalent cusps of $\Gamma$. We then derive a limit formula for $\Gamma$ at the cusp $\zeta_i = A^{-1}_i\infty$ with $A_i\in \mathrm{PSL}(2,\C)$ using the Fourier expansion of the Eisenstein series on $\Gamma$ associated to $\zeta_i$,
\[
 \lim_{s \to 1}\left(E_i(A_i^{-1}u,s) - \frac{|\Lambda_i'|\rm{vol}(\Gamma)^{-1}}{s-1} \right)   = \frac{\pi}{|\Lambda_i|}b_{ii} - \left(\frac{|\Lambda_i'|}{\mathrm{vol}(\Gamma)}+\frac{\pi}{|\Lambda_i|}\right)\log |r\cdot\eta_{\Gamma,i}(u)^2|,
\]
where $E_i(u,s)$ is the associated Eisenstein series and $\eta_{\Gamma,i}(u)$ is our analogue of the Dedekind eta function, defined in \eqref{eta}; we refer to Theorem \ref{klf} for other undefined notation. 
As one might expect, the proof of this result relies on the analytic theory of (scalar-valued) Eisenstein series associated to $\Gamma$, whose required properties have been obtained by Elstrodt, Grunewald, and Mennicke \cite{EGM1}. 

With this limit formula, we then prove a transformation property for $\eta_{\Gamma,i}(u)$, 
\[
\log\eta_{\Gamma,i}(M u) = \log\eta_{\Gamma,i}(u) + \frac12\log (|cz+d|^2+|c|^2r^2) + \pi \sqrt{-1} D_{\Gamma,i}(M),
\]
for any
\[
M = \begin{pmatrix}
 a& b \\ c & d 
\end{pmatrix} \in A_i\Gamma A_i^{-1}
\]
in Proposition \ref{transdef}, where we define the function $D_{\Gamma,i}(M)$ as a function related to generalized elliptic Dedekind sums, analogous to Goldstein's $S_{\Gamma,i}$ and $S(a,b,c,d)$ above. By Corollary \ref{hom}, it is a homomorphism from $A_i\Gamma A_i^{-1}$ to $\R$. Unfortunately, we have not yet been able to evaluate $ D_{\Gamma,i}(M)$ explicitly, so it is not clear whether it is specializes to elliptic Dedekind sums $D(c,d)$ in the case of $\Gamma = \text{PSL}(2,\mathcal O_K)$. 

\subsection{Relation to other work} 

As a point of reference, we compare our results with the work of Ito \cite{ito}. Let $\Lambda$ be a lattice in $\C$, and $\mathcal{O}_\Lambda =\{m\in\C:m\Lambda\subset \Lambda\}$ its ring of multipliers. Elliptic Dedekind sums are the analogues of Dedekind sums for PSL$(2, \mathcal O_\Lambda)$. Introduced by Sczech \cite{sczech}, they are defined as
\[
D(c,d) = \frac{1}{d}\sum_{k\in \Lambda/d\Lambda}E_1\left(\frac{ck}{d}\right)E_1\left(\frac{k}{d}\right)
\]
for $c,d\in\mathcal{O}_\Lambda$ with $d\neq0$, where $E_k(x)$
is the Eisenstein-Kronecker series. Ito constructed a function $H(u)$ analogous to the imaginary part Im$(\log\eta(z))$, such that
\[
H(Mu) = H(u) + \left(2iE_2(0)\text{Im}\Big(\frac{a+d}{c}\Big) - D(a,c)\right)
\]
when $c\neq 0$. Thus up to a constant, the expression in parentheses is analogous to our $D_{\Gamma,i}(M)$, as in our proof of Proposition \ref{transdef}. Indeed, it is known that $D(c,d)$ is purely imaginary. At the same time, it is not immediately clear whether the functions coincide, as Ito's method uses vector-valued Eisenstein series.

Elstrodt-Grunewald-Mennicke  \cite[Theorem 5.3]{EGM2} proved a first Kronecker limit formula and constructed an analogue $g(u)$ of $\log|\eta(z)|$ for $\Gamma = \mathrm{PSL}(2,\mathcal O_K)$, while Asai obtained a similar result for arbitrary number fields with class number one \cite{Asai}. We compare our limit formula with the former in Section \ref{egm}, In particular, we show that $g(u)$ can be expressed as a linear combination of our $\log|\eta_{\Gamma,i}(u)|$. We also note more recent work that studies an application of the limit formula for PSL$(2,\mathcal O_K)$ to a Jensen-Rohrlich type formula \cite{HIPT}. 

In view of classical results, one can also ask if the function $D_{\Gamma,i}$ is related to a 1-cocycle on $\Gamma$ such as shown by Sczech for $\Gamma=\mathrm{PSL}(2,\mathcal O_K)$ \cite{sczech}. It would then parametrize special values of Hecke $L$-functions of imaginary quadratic fields, as has been shown in Ito \cite{ito} and in higher rank by the last author \cite{fkw}. Moreover, as variations on Goldstein's work have recently been used to obtain numerous applications of generalized Dedekind sums on PSL$(2,\mathbb{R})$ \cite{b1,SVY,b3,b2}, it would thus be interesting to explore the analogues of these results for imaginary quadratic fields.

\section{Preliminaries} 

\subsection{Group action on the upper half space}
We first recall some basic definitions and fix notation. Let $\HH=\C \times (0,\infty) $ be the usual upper-half space model of  three-dimensional hyperbolic space. A point $u\in\HH$ can be represented by
\[
    u=(x,y,r)=(z,r)=z+rj, \qquad x,y\in\R, r>0,z=x+ \sqrt{-1} y\in\C,
\]
where $1,i,j,k$ form the usual basis of quaternions, and where we view the points $u\in\HH$ as quaternions with fourth component equal to zero. 
The group SL$(2,\C)$ acts on the upper half space $\HH$ by linear fractional transformation
\[
    u \mapsto Mu = (au+b)(cu+d)^{-1}, \quad     M = \begin{pmatrix}
     a & b\\
     c & d
    \end{pmatrix}
    \in \text{SL}(2,\C),
\]
where the right-hand side can be viewed as taking place in the skew-field of quaternions.  More explicitly, we write
\[
    Mu=z_M+r_Mj,
\]
where
\be
\label{rM}
    z_M=\frac{(az+b)(\widebar{cz+d})+a\widebar{c}r^2}{|cz+d|^2+|c|^2r^2} \text{ and }\quad r_M=\frac{r}{|cz+d|^2+|c|^2r^2}.
\ee
The projective linear group  $\mathrm{PSL}(2,\C)$ can be viewed as the group of all orientation-preserving isometries for the hyperbolic metric, and every discrete subgroup $\Gamma$ of  PSL$(2,\C)$ acts on $\HH$ discontinuously. 

Let $\infty$ be the point at infinity. Then given another point $\zeta\in \mathbb{P}^1(\C)$, there exists an element $A \in \mathrm{PSL}(2,\C)$ such that $\zeta = A^{-1}\infty$. Given any such $\zeta$, we define the stabiliser of $\zeta$ in $\Gamma$ to be 
\[
\Gamma_{\zeta} = \{M \in \Gamma : M\zeta = \zeta\}
\]
and its maximal unipotent subgroup is 
\[
\Gamma'_{\zeta} = \{M \in \Gamma_{\zeta} : \mathrm{tr}(M) = \pm 2\}.
\]
We then have that
\[
    (A \Gamma A^{-1})_{\infty}=A \Gamma_{\zeta} A^{-1}\text{ and } (A \Gamma A^{-1})'_{\infty}=A \Gamma'_{\zeta} A^{-1}.
\]
We call an element $M\in \mathrm{SL}(2,\C)$ parabolic if $M \neq \pm I$ and tr$(M)=\pm2$. Suppose that $\zeta = A^{-1} \infty$ is fixed by some parabolic element $M\in\Gamma$, $M \neq I$. Then the stabilizer of $\infty$ in $A \Gamma A^{-1}$ contains an element
\[
\begin{pmatrix}
     1 & \omega\\
     0 & 1
\end{pmatrix}, \omega\neq0,
\]
and we call $\zeta$ a cusp of $\Gamma$ if and only if the set 
\[
\Lambda := \left\{\omega: 
\begin{pmatrix}
     1 & \omega\\
     0 & 1
\end{pmatrix} 
\in A\Gamma A^{-1}\right\}
\]
is a lattice in $\C$. If $\Gamma$ is cofinite, then all fixed points of parabolic elements of $\Gamma$ are cusps, thus $\zeta$ is a cusp of $\Gamma$ when $\Gamma_{\zeta}$ contains a parabolic element. We write $\eta \equiv \zeta \Mod{\Gamma}$ if there is an element $L_0 \in \Gamma$ such that $\eta = L_0 \zeta$.

For example, let $\mathcal{O}_K$ be the ring of integers of an imaginary quadratic number field $K=\mathbb{Q}(\sqrt{D})$ for some square-free integer $D<0$, and let $h$ be the class number of $K$. Then the set of cusps of $\Gamma=\mathrm{PSL}(2,\mathcal{O}_K)$ is equal to $\mathbb{P}^1(K) = K \cup \{\infty\}$ and $\mathbb{P}^1(K)$ splits into $h$ orbits with respect to $\Gamma$.

\subsection{A Double Coset Decomposition}
We would like to express the Fourier expansion of Eisenstein series using a decomposition of $\Gamma$ into double cosets with respect to the stability groups of the cusps. In preparation, we shall describe a Bruhat-type decomposition. Let $\Gamma$ be a discrete subgroup of PSL$(2,\C)$. Chose two cusps $\zeta$ and $\eta$ of $\Gamma$ with corresponding scaling matrices $A$ and $B$ in PSL$(2,\C)$, so that $\zeta = A^{-1} \infty$ and $\eta = B^{-1} \infty$. Let $\Lambda\subset\C$ be the lattice such that 
\begin{equation}
    B \Gamma'_{\eta} B^{-1} = (B \Gamma B^{-1})'_\infty = \left\{ \begin{pmatrix}
     1 & \omega\\
     0 & 1
     \end{pmatrix}
     : \omega \in \Lambda \right\},
\end{equation}
where $\Gamma_{\zeta}$ and $\Gamma_{\eta}$ are the stability groups of the respective cusps. Let us write
\[
    \Omega_A = A\Gamma_{\zeta} A^{-1}\text{ and } \Omega_B = B\Gamma_{\eta} B^{-1}.
\]
We also define the subset the upper triangular matrices, i.e., those having a fixed point at $\infty$, denoted by
\[
\Omega_{\infty}=\left\{ \begin{pmatrix}
     \cdot & \cdot\\
     0 & \cdot
    \end{pmatrix}
    \in A\Gamma B^{-1} \right\},
\]
and also denote $\Omega_{d/c} =  \Omega_A \omega_{d/c} \Omega_B$ indexed by the matrices 
\[
    \omega_{d/c} = 
        \begin{pmatrix}
         \cdot & \cdot\\
         c & d
        \end{pmatrix}
        \in A\Gamma B^{-1}
\]
with $c \neq 0$. 

We shall partition the set $A\Gamma B^{-1}$ into double cosets with respect to these sets. 
\begin{lem}
\label{doublecoset}
Let $\zeta$, $\eta$ be cusps for $\Gamma$. Then we have the disjoint union: 
    \begin{align}
    A\Gamma B^{-1} 
    &= \delta_{\eta,\zeta}\Omega_\infty
    \, \cup \bigcup_{c \neq 0} \bigcup_{d \bmod{c\Lambda}}  \Omega_{d/c},
    \end{align}
where $\delta_{\eta,\zeta}$ equals $1$ if $\eta \equiv \zeta \Mod{\Gamma}$ and $0$ otherwise.
\end{lem}

\begin{proof}
As mentioned before, we have two cusps $\zeta$ and $\eta$ for $\Gamma$ with the corresponding scaling matrices $A$ and $B$ respectively. Suppose first that $\Omega_{\infty}$ is not empty. Then there exists an element $\omega_{\infty} = A\gamma B^{-1}$ in $\Omega_{\infty}$ with $\gamma\in\Gamma$, so that
\[
    \gamma \eta = A^{-1} \omega_{\infty} B \eta
    = A^{-1} \omega_{\infty} \infty
    = A^{-1} \infty = \zeta.
\]
This shows that the cusps $\eta$ and $\zeta$ are equivalent, the stability groups are conjugate, and $\omega_{\infty}$ is a translation. Let $\omega_1=A \gamma_1 B^{-1}$ be another element of $\Omega_{\infty}$, so that we have
\[
    \gamma \gamma_1^{-1} \zeta = A^{-1} \omega_{\infty} {\omega_1}^{-1} A \zeta
    = A^{-1} \omega_{\infty} {\omega_1}^{-1} \infty
    = A^{-1} \infty
    = \zeta.
\]
This shows that $\gamma \gamma_1^{-1} \in \Gamma_{\zeta}'$, and we have that
\[
    \omega_{\infty} {\omega_1}^{-1} = A \gamma \gamma_1^{-1} A^{-1}\in A \Gamma_{\zeta} A^{-1}.
\]
Therefore $\Omega_{\infty}$ is not empty if and only if $\eta$ and $\zeta$ are equivalent. We also note that we can write $    \Omega_{\infty} = \Omega_A \omega_{\infty} \Omega_B$ for some 
\[
    \omega_{\infty} = \begin{pmatrix}
         1 & \cdot\\
         0 & 1    \end{pmatrix}
         \in A\Gamma B^{-1},
\]
so, there exists a matrix in $\Gamma$ such that $c=0$ if and only if $\eta$ and $\zeta$ are equivalent. 

All the other elements of $AMB^{-1}$ fall into the double cosets 
\[
    \Omega_{d/c} =  \Omega_A \omega_{d/c} \Omega_B.
\]
For $\omega \in \Lambda$, we have the relation
\[
    \begin{pmatrix}
         1 & \omega'\\
         0 & 1
    \end{pmatrix}
    \begin{pmatrix}
         a & \cdot\\
         c & d
    \end{pmatrix}
    \begin{pmatrix}
         1 & \omega\\
         0 & 1
    \end{pmatrix}
    = \begin{pmatrix}
         a+c\omega' & \cdot\\
         c & d+c\omega
    \end{pmatrix}.
\]
The double cosets are thus uniquely determined by $c$, while $d$ is determined up to translation by $c\Lambda$.  For $c\neq0$, each distinct element $\omega \in \Lambda$ generates a coset from $\Omega_A\backslash A\Gamma B^{-1}$, and morevoer $\Omega_{d/c}$ is independent of the upper row of $\omega_{d/c}$.
\end{proof}

\section{Fourier expansion of Eisenstein series}

We next recall some basic facts about Eisenstein series on $\Gamma$, referring to \cite{EGM1} for details. Choose $A_1,\dots,A_h\in \mathrm{PSL}(2,\C)$ such that $\zeta_1=A_1^{-1}\infty,\dots, \zeta_h=A_h^{-1}\infty$ is a complete set of representatives of the $\Gamma$-equivalence classes of cusps of $\Gamma$. The Eisenstein series on $\Gamma$ associated to the cusp $\zeta_i$ is defined by
\[
E_{A_i}(u,s)=\sum_{M\in\Gamma'_{\zeta_i}\setminus\Gamma} (r_{A_iM})^{1+s},
\]
where $r_{A_iM}$ is defined as above. When the context is clear, we shall write $E_i(u,s) = E_{A_i}(u,s).$ The series converges absolutely and uniformly on compact subsets of $\HH\times\{s:\text{Re}\,s>1\}$ and is a $\Gamma$-invariant function on $\HH$. We note that, up to scaling, the series is independent of the choice of $A_i$ in the sense that given another $B_i\in \mathrm{PSL}(2,\C)$ such that $B_i\zeta_i=\infty$, we have that
\[
B_i = \begin{pmatrix}
 a & b \\ 0 & a^{-1}
\end{pmatrix} A_i,
\]
with $a\neq0$, and $E_{B_i}(u,s) = |a|^{2+2s}E_{A_i}(u,s)$. Moreover, if $A_i$ and $B_i$ are chosen such that $(A_i\Gamma A_i^{-1})'_\infty$ and $(B_i\Gamma B_i^{-1})'_\infty$ have fundamental parallelograms of Euclidean area one, then $|a|=1$. 

Denote by $|\Lambda|$ the area of a fundamental parallelogram associated to the lattice $\Lambda$, and $\mathrm{vol}(\Gamma)$ the covolume of $\Gamma$. For example, if $\Gamma = \text{PSL}(2,\OO_K),$ a well-known result of Humbert shows that 
\be
\label{humb}
\mathrm{vol}(\Gamma) = {|d_K|^{3/2}}\zeta_K(2)/(4\pi^2),
\ee
where $d_K$ is the discriminant of $K$ and $\zeta_K(s)$ is the Dedekind zeta function of $K$ \cite[Theorem 10.1]{EGM2}. The map $z+rj\mapsto r^{1+s}$ satisfies the differential equation
\be
\label{dr}
-\Delta r^{1+s} = (1-s^2)r^{1+s},
\ee
where 
\[
\Delta = r^2\left(\frac{\partial^2}{\partial^2 x}+\frac{\partial^2}{\partial^2 y}+\frac{\partial^2}{\partial^2 r} \right) - r\frac{\partial}{\partial r}
\]
is the usual Laplace-Beltrami operator acting on $\HH$, so by local uniform convergence it follows also that
\[
-\Delta E_i(u,s)= -(1-s)^2E_i(u,s)
\]
for Re$(s)>1$.  $E_i(u,s)$ is an automorphic function of $\Gamma$ in the sense of Definition 3.3.5 and following of \cite{EGM1}. We now recall the following properties about Eisenstein series that we shall require. For the remainder of this paper, when we say $\Gamma$ is cofinite, we shall also assume that it is not cocompact.

\begin{prop}
\label{Eis}
 Let $\Gamma$ be a cofinite discrete sugbroup of $\mathrm{PSL}(2,\C)$, and let $\zeta_i =A^{-1}_i\infty$ with associated lattice $\Lambda_i$ with dual lattice $\Lambda_i'$. Then 
 \begin{enumerate}
     \item[(i)]
 $E_{i}(u,s)$ has meromorphic continuation to the whole $s$-plane,
 \item[(ii)]
 it is holomorphic for $\mathrm{Re}(s)>0$, except for a finite number of simple poles on the real line $(0,1]$, and
 \item[(iii)]
 the simple pole at $s=1$ has residue equal to $|\Lambda_i'|\mathrm{vol}(\Gamma)^{-1}$.
 \end{enumerate}
\end{prop}

\begin{proof}
This follows from Theorems 6.1.2 and 6.1.11 \cite{EGM1}.
\end{proof}

Now, let $\zeta_i=A^{-1}_i\infty$, $\zeta_j=A^{-1}_j\infty$ be distinct cusps of $\Gamma$, let $\Lambda_j$ be the lattice in $\C$ such that 
\[
     A_j\Gamma_{\zeta_j}'A_j^{-1} = (A_j\Gamma A_j^{-1})_{\infty}' = \left\{ 
     \begin{pmatrix}
     1 & \omega\\
     0 & 1
     \end{pmatrix}
     : \omega \in \Lambda_j
     \right\},
\]
and let $\Lambda'_j$ be the dual lattice of $\Lambda_j$. The Eisenstein series $E_{i}(A_j^{-1}u,s)$ is $\Gamma$-invariant, so it can be expanded in a Fourier series in the neighbourhood of the cusp $\zeta_i$,
\[
E_i(A_j^{-1}u,s) = \sum_{\omega'\in\Lambda'_j}a_{ij,\omega'}(r,s)q^{\omega'},
\]
where $q^{\omega'} =e(\langle \omega', z\rangle),$ with $e(x) = e^{2\pi \sqrt{-1} x}$ and $\langle \cdot, \cdot \rangle$ the Euclidean scalar product in $\C$. Recall that $\zeta_i \equiv \zeta_j \Mod{\Gamma}$ if there is an element $L_0 \in \Gamma$ such that $\zeta_i = L_0 \zeta_j$, so that 
\[
    A_jL_{0}A_i =
    \begin{pmatrix}
     \cdot & \cdot\\
     0 & d_0
     \end{pmatrix}.
\]
Moreover, we define the Dirichlet series
\[
\phi_{ij,\omega'}(s) =  \sum_{c\neq 0}\sum_{d\bmod c\Lambda_j} \frac{e(\langle \omega', \frac{d}{c} \rangle)}{|c|^{2+2s}},\qquad \omega'\in\Lambda_j'.
\]
For example, if $\Gamma = \mathrm{PSL}(2,\mathbb{Z}[i])$, then $\phi_{ii,0}(s) = \zeta_K(s)/\zeta_K(s+1)$. 

The Fourier coefficients $a_{ij,\omega'}(r,s)$ have been computed explicitly in \cite{EGM2}, which we record in the following formulation, using our double coset decomposition.
\begin{lem}
Let $\mathrm{Re}(s)>1$. Then $E_i(A_j^{-1}u,s)$ has the Fourier expansion with Fourier coefficients
\begin{equation}
\label{a0}
    a_{ij,0}(r,s) = \delta_{ij} [\Gamma_{\zeta_i}:\Gamma_{\zeta_i}'] |d_0|^{-2-2s} r^{1+s} + \frac{\pi}{s|\Lambda_j|}\phi_{ij,0}(s)r^{1-s},
\end{equation}
and
\[
    a_{ij,\omega'}(r,s)= \frac{2\pi^{1+s}}{|\Lambda_j|\Gamma(1+s)} |\omega'|^{s} \phi_{ij,\omega'}(s)r K_s(2\pi|\omega'|r), \qquad \omega'\neq0,
\]
where $K_s$ denotes the modified Bessel function of the second kind. 
\end{lem}
\begin{proof}
This follows from \cite[Theorem 2.1]{EGM2} and the double-coset decomposition in Lemma \ref{doublecoset}. In particular, the coefficient $a_{ij,\omega'}(r,s)$ can  be computed as
\[
a_{ij,\omega'}(r,s) = \frac{1}{|\Lambda_i|}\sum_{M\in \Gamma'_{\zeta_i}\backslash\Gamma}\int_{Q_i}(r_{A_iMA_j^{-1}})^{1+s}e(-\langle\omega',z\rangle)dx\ dy,
\]
where $Q_i$ is a fundamental parallelogram for $\Lambda_j.$ The sum over $M$ can be rewritten as a double sum over $\omega\in\Lambda_j$ and system of representatives
$\begin{psmallmatrix}
     \cdot & \cdot\\
     c & d
     \end{psmallmatrix}$
of the double cosets in the quotient
\[
    A_i \Gamma_{\zeta_i}' A_i^{-1} \backslash A_i \Gamma A_j^{-1} / A_j \Gamma_{\zeta_j}' A_j^{-1}
\]
such that $c \neq 0$, which we can identify with the set of pairs $(c,d)$ in Lemma \ref{doublecoset}.
\end{proof}

\noindent By Proposition \ref{Eis}(i) and the explicit form of the coefficients, we see that along the line Re$(s)=1$, the nonzero coefficients $a_{ij,\omega'}(r,s)$ are holomorphic, while $a_{ij,0}(r,s)$ has meromorphic continuation according to the properties of $\phi_{ij,0}(s)$, the latter having only a simple pole at $s=1$.

Finally, we note the following formula which shall be useful to us. For $\omega'\neq0$, the Dirichlet series $\phi_{ij,\omega'}(s)$ is holomorphic in $s$, so we may evaluate the nonzero Fourier coefficients at $s=1$ as
\begin{align}
    a_{ij,\omega'} (r,1) &= \frac{2\pi^{2}|\omega'|r}{|\Lambda_i|} K_1(2\pi|\omega'|r) \phi_{ij,\omega'}(1),
\end{align}
and using the properties of the Bessel function \cite[p.66]{magnus}, namely
\[
    K_1(2\pi|\omega'|r)  = (2\pi|\omega'|r) K_0(2\pi|\omega'|r),
\]
we have
\be
\label{k0}
    a_{ij,\omega'} (r,1) = \frac{4\pi^{3}|\omega'|^2 r^2}{|\Lambda_i|} K_0(2\pi|\omega'|r) \phi_{ij,\omega'}(1).
\ee
Moreover, we note from \cite[Lemma 1-2]{Asai} that as a solution to the differential equation, the modified Bessel function $K_{2s-1}$ is an eigenfunction of $\Delta$ with eignvalue $4s(s-1)$. In particular, when $s=0$, it is annihilated by $\Delta$.

\section{The Kronecker Limit Formula}


We shall use the Fourier expansion of the Eisenstein series to derive the limit formula. It is known that $\phi_{ii,\omega'}(s)$ has a simple pole at $s=1$ when $\omega'=0$ and is continuous at $s=1$ when $\omega'\neq0$. Thus, we have that
\begin{equation}
\label{lim}
    \lim_{s \to 1} \left(E_i(A_i^{-1}u,s) - a_{ii,0} (r,s) \right) = \sum_{\substack{\omega' \in \Lambda' \\ \omega' \neq 0}} a_{ii,\omega'} (r,1) q^{\omega'}.
\end{equation}
Now we want to expand the terms of the summation above. We know that the term $a_{ii,0}(r,s)$ contains the simple pole at $s=1$ of the Eisenstein series $E_i(A_i^{-1}u,s)$. Therefore we can write the Laurent expansion around $s=1$ as 
\[
a_{ii,0}(r,s) = \frac{\alpha_{ii}}{s-1} + \beta_{ii} + O(s-1),
\]
and also
\[
\phi_{ii,0}(s) = \frac{a_{ii}}{s-1} + b_{ii} + O(s-1),
\]
where $\alpha_{ii},\beta_{ii},a_{ii}$ and $b_{ii}$ are constants independent of $s$. In the case of $\Gamma = \text{PSL}(2,\mathcal O_K)$, that  $\phi_{ii,0}(s) = \zeta_K(s)/\zeta_K(s+1)$ (c.f. \cite{Asai,Sor}). We would like to derive a more explicit formula for $\alpha_{ii}$ and $\beta_{ii}$ in terms of $a_{ii}$ and $b_{ii}$. 

\begin{lem}
\label{a0exp}
In the expansion of $a_{ii,0}(r,s)$ around $s=1$ , we have
\begin{align}
\label{alpha}
\alpha_{ii} &= {a_{ii} \pi}{|\Lambda_i|^{-1}} = |\Lambda_i'|\mathrm{vol}(\Gamma)^{-1},\\
\label{beta}
\beta_{ii}&=\delta_{ii} [\Gamma_{\zeta_i}:\Gamma_{\zeta_i}'] |d_0|^{-4} r^{2} + {\pi}{|\Lambda_i|^{-1}}(b_{ii} - (a_{ii}+1)\log(r)).
\end{align}
\end{lem}

\begin{proof}
That $\alpha_{ii} = |\Lambda_i'|\text{vol}(\Gamma)^{-1}$ simply follows from Proposition \ref{Eis}(iii). On the other hand, expanding the remaining terms of $a_{ii,0}(r,s)$ as in \eqref{a0} in a neighborhood of $s=1$, we write
\[
\frac{1}{s} = 1 - (s-1) +(s-1)^2 + O((s-1)^3),
\]
\[
r^{1+s} = r^2+r^2(s-1)\log (r)+O((s-1)^2),
\]
\[
r^{1-s}=1-(s-1)\log(r)+\frac{1}{2}(s-1)^2\log^2(r)+O((s-1)^3),
\]
and
\[
|d_0|^{-2-2s}=\frac{1}{|d_0|^{4}}-\frac{2\log(|d_0|)}{|d_0|^{4}}(s-1)+\frac{2\log^2(|d_0|)}{|d_0|^{4}}(s-1)^2+O((s-1)^3).
\]
In particular, we see that $\phi_{ii,0}(s)$ has a simple pole at $s=1$ while the remaining terms are holomorphic at $s=1$. Taking the Cauchy product of power series, we obtain for the first term $\delta_{ii} [\Gamma_{\zeta_i}:\Gamma_{\zeta_i}'] |d_0|^{-2-2s} r^{1+s}$ at $s=1$ the expansion with leading term
\[
\delta_{ii} [\Gamma_{\zeta_i}:\Gamma_{\zeta_i}'] |d_0|^{-4} r^{2} + \dots,
\]
while for the second term we have the product of $\pi/|\Lambda_i|$ with
\begin{align*}
    s^{-1} \phi_{ii,0}(s) r^{1-s}= \frac{a_{ii}}{s-1} + (b_{ii} - (a_{ii}+1)\log(r)) - b_{ii}(s-1)\log(r) +\dots
    \end{align*}
Comparing terms then gives the required formulas.
\end{proof}

We would like to separate the summation over the dual lattice $\Lambda'_j$ in \eqref{lim} into summations over a smaller region using the pair of involutions 
\[
\omega \mapsto -\omega \text{ and }\omega \mapsto \overline\omega
\]
acting on $\Lambda'_j$. Define the set
\[
\Lambda'_{j,+} = \{\omega'\in \Lambda'_j: \text{Re}(\omega')>  0\}\cup \{\omega'\in \Lambda'_j: \text{Re}(\omega') =  0, \text{Im}>0\}
\]
whereby the union of $\Lambda'_{j,+}$ and $\overline{\Lambda'_{j,+}}$ equals $\Lambda'_j\cap \mathbb C^\times$.

\begin{lem}
\label{fc}
The Fourier coefficients satisfy the following identities:
\[
a_{ij,{\omega}'}(r,1) = a_{ij,\widebar{\omega}'}(r,1),\quad and \quad a_{ij,-\omega'}(r,1) = \widebar{a_{ij,\omega'}(r,1)}.
\]
\end{lem}
\begin{proof}
We consider first the action of complex conjugation $\omega \mapsto \widebar\omega$ on the nonzero Fourier coefficient 
\[
a_{ij,\omega'}(r,1)  =  \frac{4\pi^{3}|\omega'|^2 r^2}{|\Lambda_j|} K_0(2\pi|\omega'|r) \phi_{ij,\omega'}(1).
\]
It suffices to show that $\phi_{ij,\omega'}(1) = \phi_{ij,\widebar{\omega}'}(1)$, hence $a_{ij,\omega'}(r,s)=a_{ij,\overline{\omega}'}(r,s)$. Consider
\[
\phi_{ij,\widebar{\omega}'}(1) = \lim_{s \to 1} \sum_{(c,d)} \frac{e(\langle \widebar{\omega}', \frac{d}{c}\rangle)}{|c|^{2+2s}}, 
\]
where the summation runs over distinct pairs $(c,d)$ with $c\neq 0$ and $d\bmod c\Lambda_j$. Using the property that 
\[
\langle {\widebar{\omega}}', \widebar{d/c} \rangle=\overline{\langle\omega', d/c\rangle} = \langle\omega', {d}/{c}\rangle,
\]
and the fact that if $(c,d)$ occurs in the summation then so does $(\bar{c},\bar{d})$, it follows that
\[
\sum_{(c,d)} \frac{e(\langle \widebar{\omega}', d/c\rangle)}{|c|^{2+2s}} = \sum_{(\bar{c},\bar{d})} \frac{e(\langle \widebar{\omega}', \widebar{d/c}\rangle)}{|c|^{2+2s}} = \sum_{(c,d)} \frac{e(\langle {\omega}', d/c\rangle)}{|c|^{2+2s}} ,
\]
and therefore $\phi_{ij,\omega'}(1) = \phi_{ij,\widebar{\omega}'}(1)$.

Next consider the map $\omega\mapsto -\omega$. We observe that $e(\langle -{\omega}', d/c\rangle) = \widebar{e(\langle {\omega}', d/c\rangle)}$, so that 
\[
\phi_{ij,-\omega'}(1) = \overline{\phi_{ij,\omega'}(1)}.
\]
 Also, using the integral representation of the modified Bessel function 
 \[
 K_0(x) = \int_0^\infty\frac{\cos(xt)}{\sqrt{t^2+1}}dt,
 \]
 it follows that $K_0(2\pi|\omega'|r) =\widebar{K_0(2\pi|\omega'|r)}$, and thus together with \eqref{k0} we have
\[
a_{ij,-\omega'}(r,1) = \widebar{a_{ij,\omega'}(r,1)},
\]
giving the second identity.
\end{proof}

Now we are ready to prove the analogue of Kronecker's first limit formula for the Eisenstein series $E_{i}(A_i^{-1}u,s)$. For ease of notation, let us define the constant
\[
C_{\Gamma,i} = \frac{|\Lambda_i'|}{\mathrm{vol}(\Gamma)}+\frac{\pi}{|\Lambda_i|}.
\]
Then combining the previous results and specializing to $i=j$, we obtain the following formula.
\begin{thm}
\label{klf}
With notation as above, we have
\begin{equation}
\label{klf1}
 \lim_{s \to 1}\left(E_i(A_i^{-1}u,s) - \frac{|\Lambda_i'|\rm{vol}(\Gamma)^{-1}}{s-1} \right)  = \frac{\pi}{|\Lambda_i|}b_{ii} - C_{\Gamma,i}\log |r\cdot\eta_{\Gamma,i}(u)^2|,
 \end{equation}
where the function $\eta_{\Gamma,i}(u)$ is defined by
\be
\label{eta}
C_{\Gamma,i}\log\eta_{\Gamma,i}(u) = \frac12[\Gamma_{\zeta_i}:\Gamma_{\zeta_i}'] |d_0|^{-4} r^{2} +  \sum_{\omega'\in\Lambda'_{i,+}} a_{ii,\omega'}(r,1)q^{\omega'}.
\ee
\end{thm}

\begin{proof}
The convergence of the righthand side of \eqref{klf1} follows from bounding the nonconstant Fourier coefficients of the Eisenstein series $E_i(A_i^{-1}u,s)$. In particular, by \cite[(6.3.21)]{EGM1}, or more straightforwardly by observing that the nonconstant Fourier coefficients are holomorphic on Re$(s)=1$  and have exponential decay by virtue of the Bessel functions, we have that 
\[
E_i(A_i^{-1}u,s)- a_{ii,0}(r,s) = O(e^{-Kr})
\]
as $r\to\infty$ for some fixed $K>0$ whenever $s$ is not a pole of $E_i(\cdot,s)$. 

We now examine the limit. Using the expansion of $a_{ij,0}(r,s)$, we can rearrange \eqref{lim} as
\[
\lim_{s \to 1} \left(E_i(A_j^{-1}u,s) - \frac{\alpha_{ij}}{s-1} \right) = \beta_{ij}  + \sum_{\substack{\omega'\in\Lambda'_i\\\omega'\neq0}}a_{ij,\omega'}(r,1)q^{\omega'}.
\]
Using Lemma \ref{a0exp} and setting $i=j$, the righthand side is equal to
\be
\label{sum}
\frac{\pi}{|\Lambda_i|}(b_{ii} - \left(a_{ii}+1\right)\log(r)) +  [\Gamma_{\zeta_i}:\Gamma_{\zeta_i}'] |d_0|^{-4} r^{2} + \sum_{\substack{\omega'\in\Lambda'_i\\\omega'\neq0}}a_{ii,\omega'}(r,1)q^{\omega'}.
\ee
Since $C_{\Gamma,i} = \pi(a_{ii}+1)/|\Lambda_i|,$ the first term can be written as 
\[
{\pi}{|\Lambda_i|^{-1}}b_{ii} - C_{\Gamma,i}\log(r).
\]
On the other hand, the map $\omega'\mapsto-{\omega}'$ is symmetric about the origin, so by Lemma \ref{fc}(ii) we can separate the series in \eqref{sum} into summations over $\Lambda'_{i,+}$. Writing for $a(\omega') = a_{ii,\omega'}(r,1)$, we have
\begin{align*}
\sum_{\text{Re}(\omega')\ge 0}a(\omega')q^{\omega'}   + \sum_{\text{Re}(\omega')< 0}a(\omega')q^{\omega'} 
&= \sum_{\text{Re}(\omega')\ge 0}a(\omega')q^{\omega'}   + \sum_{\text{Re}(\omega')> 0}a(- \omega')q^{- \omega'}\\
&= \sum_{\text{Re}(\omega')\ge 0}a(\omega')q^{\omega'}   + \sum_{\text{Re}(\omega')> 0}\overline{a(\omega')q^{ {\omega'}}},
\end{align*}
where in the last line we have used the fact that the second  sum is invariant under  $\omega'\mapsto \overline{\omega'}$. Also, the contribution of Re$(\omega')=0$ in the first sum can be written as 
\[
\sum_{\text{Re}(\omega') = 0}a(\omega')q^{\omega'} = \sum_{\substack{\text{Re}(\omega') = 0\\ \text{Im}(\omega') >0 }}a(\omega')q^{\omega'}  + \sum_{\substack{\text{Re}(\omega') = 0\\ \text{Im}(\omega') >0 }}\overline{a(\omega')q^{\omega'}} 
\]
since we are assuming $\omega'$ is nonzero. Altogether, we have therefore
\[
\sum_{\substack{\omega'\in\Lambda'_i\\\omega'\neq0}}a_{ii,\omega'}(r,1)q^{\omega'} = \sum_{\omega\in \Lambda_{j,+}'} a(\omega')q^{\omega'}   + \sum_{\omega\in \Lambda_{j,+}'}\overline{a(\omega')q^{ {\omega'}}}.
\]
It then follows from the definition of $\log \eta_{\Gamma,i}$ in \eqref{eta} that
\begin{align*}
&\frac{\pi}{|\Lambda_i|}b_{ii} - C_{\Gamma,i}\left(\log(r) + \log\eta_{\Gamma,i}(u) + \overline{\log\eta_{\Gamma,i}(u)}\right),
\end{align*}
and the required formula follows.
\end{proof}

\section{The function $D_{\Gamma,i}$}

We now examine the properties of our function $\log\eta_{\Gamma,i}$, and use it to construct a $D_{\Gamma,i}$ analogous to the classical case. Recall that a function $f:\HH\to \C$ is called harmonic if it is twice-differentiable and $\Delta f = 0$.  

\begin{lem}
\label{harm}
$\log\eta_{\Gamma,i}$ is harmonic.
\end{lem}

\begin{proof}
Differentiating term by term, by \eqref{dr} the constant term is annihilated by $\Delta$. For the higher Fourier coefficients, this follows as in \cite[Lemma 1-2]{Asai} (see also \cite[(20)]{ito}). Namely, the modified Bessel functions $K_s(r)$ satisfy the differential equation
\[
r^2\frac{d^2}{dr^2}K_s(r) + r\frac{d}{dr}K_s(r)-(r^2+s^2)K_s(r)=0,
\]
and following \cite[p.77]{magnus} the function $w = r K_s(2\pi |\alpha| r)$ then satisfies the differential equation
\[
r^2\frac{d^2w}{dr^2} - r\frac{dw}{dr}+((2\pi |\alpha| r)^2+(1- s^2 ))w=0.
\]
On the other hand, applying $\Delta$ to the relevant part of the Fourier coefficient $h_\alpha(r) = r K_s(2\pi|\omega'|r) e(\langle \omega', z\rangle)$ yields
\[
\left(r^2\frac{d}{dr^2}  - r\frac{d}{dr} + (1-s^2) - 4\pi |\alpha|^2 r^2 \right)h_\alpha(r) = 0.
\]
which agrees with the preceding equation. Specializing at $s=1$ then yields the required property.
\end{proof}

Using the analytic properties of the function $\log\eta_{\Gamma,i}$, we can define the analogue of $S(a,b,c,d)$ in \eqref{trans}, now associated to the Kleinian group $\Gamma$ at the cusp $\zeta_i$.  The following formula can be viewed as a generalization of \eqref{trans} to our setting.

\begin{prop}
\label{transdef}
For any $M \in A_i\Gamma A_i^{-1}$, we have
\[
\log\eta_{\Gamma,i}(M u) = \log\eta_{\Gamma,i}(u) + \frac12\log (|cz+d|^2+|c|^2r^2) + \pi \sqrt{-1} D_{\Gamma,i}(M),
\]
where $D_{\Gamma,i}(M)$ is a real-valued function depending on $\Gamma, i,$ and $M$.
\end{prop}

\begin{proof}
By Theorem \ref{klf}, the fact that $E_i(u,s)$ is automorphic in $u$, and $\alpha_{ii},\beta_{ii}$ are independent of $u$, we see that
\[
\log | r_{M} \eta_{\Gamma,i}(Mu)^2| = \log |r\eta_{\Gamma,i}(u)^2|.
\]
On the other hand, from \eqref{rM} we have
\[
\log | r_{M} \eta_{\Gamma,i}(Mu)^2| = \log r  - \log (|cz+d|^2+|c|^2r^2)+\log |\eta_{\Gamma,i}(Mu)|^2,
\]
and it follows that
\be
\label{logtrans}
\log |\eta_{\Gamma,i}(Mu)| = \log |\eta_{\Gamma,i}(u)| + \frac12\log (|cz+d|^2+|c|^2r^2).
\ee
Therefore the function defined by
\[
F(M,u) = \log \eta_{\Gamma,i}(Mu) - \log \eta_{\Gamma,i}(u) - \frac12\log (|cz+d|^2+|c|^2r^2)
\]
has real part identical to zero. We thus conclude that $F(M,u)$ is equal to a purely imaginary constant that depends on at most $\Gamma, i,$ $M$, and $u$. 

As for the dependence on $u$, consider the imaginary part
\[
\text{Im}(F(M,u)) = \text{Im}(\log \eta_{\Gamma,i}(Mu)) - \text{Im}(\log \eta_{\Gamma,i}(u)).
\]
It follows from Lemma \ref{harm} and the action of $\Delta$ on the righthand side of \eqref{klf1} that $\log|\eta_{\Gamma,i}(u)|$  is harmonic and therefore so is $\text{Im}(\log \eta_{\Gamma,i}(u))$. We see also from \eqref{klf1} that $\text{Im}(\log \eta_{\Gamma,i}(u))$ is $\Gamma$-invariant and has rapid decay at cusps. Then using spectral properties of the Laplace operator, namely that fact any $\Gamma$-invariant harmonic function that square integrable over a fundamental domain of $\Gamma$ is constant  \cite[Corollary 3.3.4]{EGM1}, we conclude that $\text{Im}(\log \eta_{\Gamma,i}(u))$ is independent of $u$.\footnote{We thank the reviewer for suggesting this proof of independence.}
\end{proof}


Following \cite[\S3.3]{EGM1}, we define an automorphic function on $\HH$ for the group $\Gamma$ with parameter $\lambda \in\C$ to be a $\Gamma$-invariant, twice-differentiable, complex-valued function $f$ on $\HH$ satisfying  $-\Delta f = \lambda f$ and $f(A^{-1}u)=O(r^K)$ for some fixed constant $K>0$ as $r\to\infty$, uniformly with respect to $z\in\C$.

\begin{cor}
\label{hom}
$\log|\eta_{\Gamma,i}(u)|$  is a harmonic automorphic function with respect to $A_i\Gamma A_i^{-1}$ of weight $\frac12$ with automorphy factor $ \log (|cz+d|^2+|c|^2r^2)$, and the map
\[
D_{\Gamma,i}: A\Gamma_i A^{-1} \to \mathbb{R}
\]
is a homomorphism.
\end{cor}

\begin{proof}
The transformation property follows from Proposition \ref{transdef}, whereas the harmonicity follows the same argument as Lemma \ref{harm}. Finally, let $M,N\in A\Gamma_i A^{-1}$ and write
\[
M = \begin{pmatrix}a & b \\ c & d\end{pmatrix}, \quad N = \begin{pmatrix}a' & b' \\ c' & d'\end{pmatrix}, \quad MN = \begin{pmatrix}a'' & b'' \\ c'' & d''\end{pmatrix}.
\]
Then applying Proposition \ref{transdef} we can expand $\log\eta_{\Gamma,i}(MN u)$ on the one hand as 
\[
 \log\eta_{\Gamma,i}(u) + \frac12\log (|c''z+d''^2|+|c''|^2r^2) + \pi \sqrt{-1} D_{\Gamma,i}(MN),
\]
and on the other hand as
\begin{align*}
& \log\eta_{\Gamma,i}(Nu) + \frac12\log (|cz_M+d|^2+|c|^2r_M^2) + \pi \sqrt{-1} D_{\Gamma,i}(M)\\
=\ & \log\eta_{\Gamma,i}(u) + \frac12\log (|cz_M+d|^2+|c|^2r_M^2) + \frac12\log (|c'z+d'|^2+|c'|^2r^2) \\
&\qquad + \pi \sqrt{-1} D_{\Gamma,i}(M) + \pi \sqrt{-1} D_{\Gamma,i}(N).
\end{align*}
Then using the property of the group action of $\Gamma$ on $\HH$ or by direct computation, one checks that 
\begin{align*}
&\log (|cz_M+d|^2+|c|^2r_M^2) + \log (|c'z+d'|^2+|c'|^2r^2) \\
=\ &\log \left(\left|c\frac{(az+b)(\widebar{cz+d})+a\widebar{c}r^2}{|cz+d|^2+|c|^2r^2} +d\right|^2+\left(\frac{|c|r}{|cz+d|^2+|c|^2r^2}\right)^2\right)\\
& + \log\left(|c'z+d'|^2+|c'|^2r^2\right) \\
= \ & \log (|c''z+d''|^2+|c''|^2r^2),
\end{align*}
using the fact that $c'' = ca'+dc'$ and $d'' = cb' + dd'$. We thus conclude that 
\[
D_{\Gamma,i}(MN) =D_{\Gamma,i}(M)  + D_{\Gamma,i}(N),
\]
and the homomorphism property follows.
\end{proof}

We record some simple properties  that follow as  immediate consequences of the homomorphism property.

\begin{cor}
The function $D_{\Gamma,i}$ satisfies
\begin{itemize}
\item[(i)] $D_{\Gamma,i}(M) = - D_{\Gamma,i}(M^{-1})$
\item[(ii)] $D_{\Gamma,i}(M) = D_{\Gamma,i}(N)$ if $M,N\in\Gamma$ are conjugate.
\end{itemize}
\end{cor}

\noindent We note that the first property can be seen as the analogue of the identity 
\[
D(c,d) = -D(\bar{c},\bar{d})
\]
satisfied by the usual elliptic Dedekind sums \cite[Theorem 1]{ito2}.

\section{Comparison with the limit formula of a modified Eisenstein series}
\label{egm}
For the remainder of this paper, we fix $\Gamma = \text{PSL}(2,\mathcal O_K)$. In this case, we can compare our results with the limit formula obtained in \cite[Theorem 5.3]{EGM2}, where the function $g(u)$ of (5.17) in the reference is analogous to the classical $\log|\eta(z)|$, whereas our $\eta_{\Gamma,i}(u)$ is defined so as to obtain a function without the absolute value. 
To explain the relation, we first introduce some notation. Let $\mathcal{M}$ be the set of all fractional ideals of $\mathcal O_K$. Then for any $\mathfrak m\in\mathcal M, u \in \mathbb H, s\in \mathbb C$ with Re$(s)>1$, define 
\[
\hat E_{\mathfrak m}(u,s) =  N(\mathfrak m)^{1+s}\sum_{\substack{c,d\in\mathfrak m\\ (c,d)\neq (0,0)}}\left(\frac{r}{|cz+d|^2+|c|^2r^2}\right)^{1+s}
\]
and 
\[
\zeta(\mathfrak m, \mathfrak n, s) = N(\mathfrak m\mathfrak n^{-1})^s\sum_{\substack{\lambda\in \mathfrak m\mathfrak n^{-1}\\ \lambda\neq0}}N(\lambda)^{-s}.
\]
The latter zeta function was studied earlier also by Lang \cite{lang}. Both functions depend only on the equivalence class of $\mathfrak m $ in the ideal class group $\mathcal M/K^*$. By \cite[Theorem 8.3.1]{EGM1}, $\hat E_{\mathfrak m}(u,s)$ is holomorphic for all $s$, except for a simple pole at $s=1$ with residue $4\pi^2/|d_K|$
They relate to our earlier Eisenstein series by the relations
\be
\label{Eu1}
 E_{\mathfrak u_A}(u,s) = 2N(\mathfrak u_A)^{1+s}E_A(u,s) 
\ee
and
\be
\label{Eu2}
\hat E_{\mathfrak m}(u,s) = \frac{1}{|\mathcal O^*_K|}\sum_{\mathfrak u_A\in \mathcal M/K^*}\zeta(\mathfrak m,\mathfrak u_A,s+1)E_{\mathfrak u_A}(u,s),
\ee
 by Proposition 8.1.4 and Theorem 8.1.7 of \cite{EGM1}, where for $A\zeta = \infty$ and $A = (\begin{smallmatrix}a & b \\ c & d\end{smallmatrix})$, $\mathfrak u_A$ is the ideal generated by $c,d$ as an $\mathcal O_K$-module.
Note that the number of cusps is equal to the class number $h$ of $K$, which agrees with the length of the sum.

Now the limit formula given in \cite[Theorem 8.3.3]{EGM1} takes the form 
\[
\lim_{s\to1}\left(\hat E_{\mathfrak m}(A_i^{-1}u,s)-\frac{4\pi^2}{|d_K|}\frac{1}{s-1}\right) = \frac{4\pi^2}{|d_K|}g(u) + \dots
\]
where as in \cite[(8.3.17)]{EGM1},
\begin{align*}
g(u) &= \frac{|d_K|}{4\pi^2}\zeta(\mathfrak m,\mathfrak u,2)N(\mathfrak u)^2r^2 \\
&+ 2N(\mathfrak u)\sum_{0\neq w\in\mathfrak u^2}|w|\sigma_{-1}(\mathfrak m,\mathfrak u,w)r K_1\left(\frac{4\pi|w|r}{\sqrt{|d_K|}}\right)e^{2\pi i \langle 2\bar w|d_K|^{-1/2},z\rangle}
\end{align*}
with
\[
\sigma_{s}(\mathfrak a,\mathfrak b,w) = N(\mathfrak a)^{-s}\sum_{\substack{\lambda\in\mathfrak a\mathfrak b\\ w\in\lambda\mathfrak a^{-1}\mathfrak b}}N(\lambda)^s.
\]
The function $g(u)$ is a harmonic function satisfying the transformation property
\[
g(Mu) + \log (|cz+d|^2+|c|^2r^2) = g(u)
\]
parallel to \eqref{logtrans}. In comparison,  the term that would have been subtracted from our Eisenstein series in our formulation in Theorem \ref{klf} would be, by \eqref{humb}, 
\[
 \lim_{s \to 1}\left(E_i(A_i^{-1}u,s) - \frac{4\pi^2|\Lambda_i'|}{|d_K|^{3/2}\zeta_K(2)}\cdot\frac{1}{s-1}\right) = -2C_{\Gamma_i}\log|\eta_{\Gamma_i}(u)|+\dots.
\]
We can then express $g(u)$ as a linear combination of $\log|\eta_{\text{PSL}(2,\mathcal O_K),i}(u)|$. 

\begin{prop}
\label{geta}
Let $\Gamma = \mathrm{PSL}(2,\mathcal O_K)$, and let $\mathfrak u_i = \mathfrak u_{A_i}, i=1,\dots,h$ be a complete set of representatives of  $ \mathcal M/K^*$.  We have
\[
g(u) =  -\frac{|d_K|}{\pi^2|\mathcal O^*_K|}\sum_{i=1}^hN(\mathfrak u_i)^{2}\zeta(\mathfrak m,\mathfrak u_i,2) C_{\Gamma,i}\log|\eta_{\Gamma,i}(u)| + B_\Gamma(r),
\]
where $B_\Gamma(r)$ is the function \eqref{Bg} depending on $\Gamma$ and $r$.
\end{prop}

\begin{proof}
From the preceding discussion, we relate the two Eisenstein series by \eqref{Eu1} and \eqref{Eu2},
\be
\label{eisid}
\hat E_{\mathfrak m}(u,s) = \frac{1}{|\mathcal O^*_K|}\sum_{\mathfrak u_A\in \mathcal M/K^*}2N(\mathfrak u_A)^{1+s}\zeta(\mathfrak m,\mathfrak u_A,s+1) E_A(u,s).
\ee
The righthand sum is finite and therefore the poles at $s=1$ of both sides must coincide as meromorphic functions in $s$. Thus subtracting the leading terms on either side of the equation, we may compare the limit formulas to deduce the identity. The function  $B_\Gamma(r)$ is given by the difference of the term ${\pi}b_{ii}/|\Lambda_i|$ in \eqref{klf1} and the remaining terms in \cite[Theorem 8.3.3]{EGM1}, namely
\be
\label{Bg}
\frac{4\pi^2}{|d_K|}(2\gamma - 1 - \log|d_K|-\log (r ) - \frac16 \log \tilde{g}(\mathfrak{m})),
\ee
where $\tilde g(\mathfrak m) = (2\pi)^{-12}N(\mathfrak m)^6|\Delta(\mathfrak m)|$ and $\Delta = g^3_2 - 27g^2_3$ is the discriminant from the theory of elliptic functions. Also, $\gamma$ is the usual Euler-Mascheroni constant.
\end{proof}

As a byproduct, we have the following expression for the special value of the Dedekind zeta function.

\begin{cor}
With notation as above, we have
\[
\zeta_K(2)= |d_K|^{-1/2} \sum_{i=1}^h  {|\Lambda_i'|} = |d_K|^{-1/2} \sum_{i=1}^h  N(\mathfrak u_{i})^{2}.
\]
\end{cor}
\begin{proof}
This  follows directly from comparing residues on either side of \eqref{eisid} at $s=1$,
\[
\frac{4\pi^2}{|d_K|} = \sum_{i=1}^h \frac{4\pi^2|\Lambda_i'|}{|d_K|^{3/2}\zeta_K(2)}.
\]
The second equality follows from the identity
$
|\Lambda_i'| = N(\mathfrak u_{i})^{2}
$
as an application of (8.2.16) in \cite{EGM1}, 
$
\Lambda'_i = 2|d_K|^{-1/2}\bar{\mathfrak{u}}^2_i,
$
and the fact that $|\mathcal O_K^*| = \frac12|d_K|^{-1/2}$.
\end{proof}

\noindent As a side note, we remark that one also has $|\Lambda_i| = \frac12|d_K|^{1/2}N(\mathfrak u_i)^{-2}$ by \cite[(8.2.12)]{EGM1}.

\subsubsection*{Acknowledgments}
The first two authors were supported by the Smith College Summer Research Fellowship (SURF) Program. The third author was partially supported by NSF grant DMS-2212924, and thanks Jorge Fl\'orez for helpful discussions related to the paper. The authors are grateful to the referee for careful readings of the paper, and helpful comments that greatly strengthened the results of the paper.



\end{document}